\documentclass[12pt]{article}
\usepackage[utf8]{inputenc} 
\usepackage{amsmath,amsfonts,amssymb,amsthm,amstext}
\usepackage{graphicx}
\newcommand{\pic}[4]{\vspace{1ex}\setlength{\unitlength}{.05\textwidth}
\begin{picture}(0,0)(0,0)
\put(#2){\includegraphics[#4]{#1.jpg}} 
\end{picture}\vspace*{#3cm}\vspace{1ex}} 

\newtheorem{theo}{Theorem}[section]
\newtheorem{lemm}[theo]{Lemma}
\newtheorem{coro}[theo]{Corollary}
\newtheorem{prop}[theo]{Proposition}

\theoremstyle{definition}
\newtheorem{defi}[theo]{Definition}
\newtheorem{remark}[theo]{Remark}
\newcounter{example}
\newcommand{\example}{\addtocounter{example}{1}\smallskip\noindent\textbf{Example \arabic{example}. }}

\def\C{\mathbb{C}}
\def\R{\mathbb{R}}

\def\chcos{\sqrt{\cosh\eta-\cos\theta}}
\def\chcosz{\sqrt{\cosh\eta_0-\cos\theta}}
\def\dom{\Omega_{\eta_0}}

\def\bdry{{\partial\dom}}
\def\opt{\textrm{opt}}

\newcommand{\ind}[2]{{}^{\hspace{.1ex}#2}_{#1}}
\newcommand{\harint}[2]{I\ind{#1}{#2}}
\newcommand{\harext}[2]{E\ind{#1}{#2}}
\DeclareMathOperator{\grad}{{\mbox{\rm grad}}}
\DeclareMathOperator{\nor}{{\mbox{\rm nor}}}
\DeclareMathOperator{\har}{{\mbox{\rm Har}}}

\begin{document}

\begin{center}
  \parbox{.8\textwidth}{\begin{center}
      \large NEUMANN PROBLEM ON A TORUS \\
    Z.\ Ashtab, J.\ Morais, R.\ M.\ Porter
  \end{center} }
 
\parbox{.8\textwidth}
 {Department of Mathematics, CINVESTAV-Quer\'etaro, Libramiento Norponiente \#2000, Fracc.~Real de Juriquilla. Santiago de Quer\'etaro, 
C.P.~76230 Mexico.}

 \parbox{.8\textwidth}
{Department of Mathematics, ITAM, R\'io Hondo~\#1, Col.~Progreso Tizap\'an, Mexico City, C.P.~01080 Mexico. }
\end{center}

\begin{center}
  \parbox{.9\textwidth}{\textbf{Abstract.} We
      consider the Dirichlet-to-Neumann mapping and the Neumann
      problem for the Laplace operator on a torus, given in
      toroidal coordinates. The Dirichlet-to-Neumann mapping is
      expressed with respect to series expansions in toroidal
      harmonics and thereby reduced to algebraic manipulations on the
      coefficients. A method for computing the numerical solutions of
      the corresponding Neumann problem is presented, and numerical
      illustrations are provided. We combine the results for interior
      and exterior domains to solve the Neumann problem for a toroidal
      shell.}
\end{center}

 

\section{Introduction}

The Dirichlet-to-Neumann map for the Laplace equation plays an
important role in various areas of analysis (e.g., elliptic boundary
value problems \cite{DK1987,Lannes2013,McLean2000,Raynor1934}, inverse problems
\cite{Calderon1980,Isakov1998}) and physics (e.g., electromagnetism
\cite{Caldwell1984}, electrical transmission \cite{DZ2009}, fluid
mechanics \cite{CraigSulem1993}, electrical impedance tomography
\cite{IMNS2006,KV1985,SU1990}). The map describes the relationship
between the value of a function
$f\colon \partial \Omega \to \mathbb{R}$ (Dirichlet datum) on the
boundary $\partial \Omega$ of some spatial domain $\Omega$ and the
boundary normal derivative (Neumann datum) of the unique harmonic
extension $u\colon \Omega \to \mathbb{R}$ determined by having the
boundary values $f$. The Neumann problem is to recover $f$ (or $u$)
from the normal derivative.

When $\Omega$ is a torus, harmonic functions defined on $\Omega$ are
naturally expressed as series based on a doubly-indexed collection of
\textit{toroidal harmonics} (involving half-integer associated
Legendre functions of the first and second kinds) \cite{Hob1931}, which
are orthogonal with respect to a certain weighted $L^2$-inner product
over a torus. (We have seen applications \cite{Caldwell1984} where the
potential in a toroidal conductor is modeled in spherical coordinates,
which are not ideally suited for such problems.)  The coefficients of
this expression provide a solution to the Dirichlet problem for the
Laplacian quite directly. However, unlike the case for a sphere, the
Dirichlet-to-Neumann mapping for a torus turns out to be much more
complicated, and the solution of the Neumann problem involves solving
an infinite system of linear equations. We express the well-known necessary and
sufficient condition for the solvability of the Neumann problem
(compatibility condition), as well as the normalization condition, in
terms of the Fourier coefficients. The solution to the Neumann problem
turns out to involve a special twist in that the free parameter in the
undetermined linear system cannot be found algebraically, as far as we
know. Therefore we express it as a limit of easily calculated
algebraic expressions.  The analysis is illustrated through  
numerical examples.

The paper is organized as follows. Properties of toroidal harmonic
functions are summarized in Section
\ref{Toroidal_coordinates_and_toroidal_harmonics_Section}. In Section
\ref{Dirichlet-to-Neumann_mapping_Section}, we compute the toroidal
Neumann derivative, from which we deduce the Dirichlet-to-Neumann
mapping.  The expansion coefficients of the normal derivative are
linear expressions in the Fourier coefficients on the surface of the
torus. While the mapping exists in the context of the appropriate
Sobolev function spaces associated with the Laplace equation and the
boundary data, as well as the trace operator,
we only justify the derivation of our formulas for the normal
derivative under stronger smoothness assumptions.  Section
\ref{Neumann_problem} presents the algorithm and numerical examples to
show the accuracy of the procedure. Section
\ref{Exterior_toroidal_domain_and_toroidal_shells} combines the
results for interior and exterior domains to solve the Neumann problem
for a toroidal shell.

%%%%%%%%%%%%%%%%%%%%%%%%%%%%%%%%%%%%%%%%%%%%%%%%%%%%%%%%%%%%%
%%%%%%%%%%%%%%%%%%%%%%%%%%%%%%%%%%%%%%%%%%%%%%%%%%%%%%%%%%%%%

\section{Toroidal coordinates and toroidal
  harmonics} \label{Toroidal_coordinates_and_toroidal_harmonics_Section}

In this section, we introduce notation and summarize several well-known facts to be used throughout the paper.

\subsection{Normal derivatives}

One defines toroidal coordinates $(\eta,\theta,\varphi)$ for
a point $x=(x_0,x_1,x_2)$ in three-dimensional Euclidean space by the
relations
 \begin{align}
  x_0 &= \frac{\sin \theta}{\cosh\eta-\cos\theta}, \ \
  x_1 = \frac{\sinh \eta \, \cos \varphi}{\cosh\eta-\cos\theta}, \ \  
   x_2 = \frac{\sinh \eta \, \sin \varphi}{\cosh\eta-\cos\theta} 
 \end{align}
 in the range $\eta\in(0,\infty)$, $\theta\in[-\pi,\pi]$,
 $\varphi\in(-\pi,\pi]$. For a geometric explanation of this
 coordinate system, see \cite{Hob1931,MoonSpencer1988}.  The
 correspondence is singular on the two subsets
 %\begin{align*}
 %  S^1=\{x\in\R^3\colon\ x_0=0,\ x_1^2+x_2^2=1\},\quad
 %  \R_0=\{x\in\R^3\colon\ x_1=x_2=0\}, 
 %\end{align*}
 $S^1=\{x\in\R^3\colon\ x_0=0,\ x_1^2+x_2^2=1\}$ and
 $\R_0=\{x\in\R^3\colon\ x_1=x_2=0\}$, 
 which correspond
 respectively to the limiting cases $\eta\to\infty$ and $\eta\to0$.
 For any fixed $\eta_0>0$, these coordinates define the interior and
 exterior toroidal domains
 \begin{align}
   \dom = \{x\colon\ \eta>\eta_0\} \cup S^1 ,\quad
  \dom^* = \{x\colon\ \eta<\eta_0\} \cup \R_0.
 \end{align}
Any open solid torus in $\R^3$ can be shifted and rescaled to a torus of the form
$\dom$.

By calculating the coordinate tangent vectors
$x_\eta = \frac{\partial x}{\partial \eta}$, $x_\theta = \frac{\partial x}{\partial \theta}$, $x_\varphi = \frac{\partial x}{\partial \varphi}$ 
and normalizing $x_\theta\times x_\varphi$, one obtains the normal unit vector 
\begin{align} \label{eq:unitnormal}
  {\mathbf{n}}  = \frac{-1}{(\cosh\eta_{0}-\cos\theta)} \big(&
    \sinh\eta_{0}\sin\theta, \ \cos\varphi(\cosh\eta_{0}\cos\theta-1),\ \nonumber \\
   &  \
     \sin\varphi(\cosh\eta_{0}\cos\theta-1) \big) .
\end{align} 
Since  toroidal coordinates form an orthogonal coordinate system, we have
${\mathbf{n}} = x_\eta/\vert x_\eta\vert$, and recalling that $\eta\to\infty$
at $S^1 \subseteq \dom$, we see that ${\mathbf{n}}$ is in fact the
inward pointing normal vector on $\bdry$.

The \textit{normal derivative} of a function $f$ defined in a
neighborhood $V$ of a point $x\in\bdry$ (or a half-neighborhood
$\nor f(x) =  \frac{d}{dt}f(x+t\,{\mathbf{n}}) \vert_{t=0^+}$
$= (\grad f(x))\cdot{\mathbf{n}}$,
and by orthogonality of the coordinate system
the normal derivative is also equal to
the following, which is often more convenient for calculations:
\begin{align} \label{eq:norfddeta}
  \nor f (x)=  \frac{1}
  {\vert x_\eta\vert}\,\frac{d}{d\eta}f(\eta,\theta,\varphi)\bigg\vert_{\eta=\eta_0^+} .
\end{align}
In using the notation $\nor f$, the fixed value of $\eta_0$ will always be
understood.

\subsection{Toroidal harmonics}

The associated Legendre functions (Ferrer's functions) of the first
and second kinds for $t>1$ are defined for integer values of
$n,m\ge0$, respectively, as
\begin{align}
  P_n^m(t) &=(t^{2}-1)^{m/2}\frac{d^{m}P_{n}(t)}{dt^{m}},  \nonumber\\
  Q_n^m(t) &=\frac{1}{2}P_{n}(t)\log\frac{t+1}{t-1}-\sum_{k=0}^{n-1}\frac{P_{k}(t)P_{n-k-1}(t)}{t-k}, \label{eq:legendrefunctions}
\end{align}
where $P_n(t)$ denotes the classical Legendre polynomial of degree $n$
\cite{AbramowitzStegun1964,ArfkenWeberHarris2013,Bateman1944,CourantHilbert1953,Erdelyi1953,GradshsteynRyzhik2007,Hob1931,JeffreyDai2008,Snow1952,WhittakerWatson1927}. When
one extends these functions analytically in the complex plane away
from the ray $t\in(1,\infty)$, they are branched at $t=\pm1$. However, they
are entire functions of $n$ and $m$ regarded as complex variables
\cite{Hob1931}. In this sense \eqref{eq:legendrefunctions} can be
taken as a definition of the associated Legendre functions for
half-integer values of $n$ as we will need here.

 We will abbreviate $\Phi\ind{n}{\nu}(\theta) = \cos n\theta$ for $\nu=1$, and
$\Phi\ind{n}{\nu}(\theta) = \sin n\theta$ for $\nu=-1$. The
\textit{interior toroidal harmonic functions} are
\begin{align} \label{eq:interiorharmonic}
  I\ind{n,m}{\nu,\mu}(x)=I\ind{n,m}{\nu,\mu}[\eta_{0}](x) =
  \chcos \,\frac{Q\ind{n-1/2}{m}(\cosh\eta)}{Q\ind{n-1/2}{m}(\cosh\eta_{0})} \,
   \Phi\ind{n}{\nu}(\theta) \,  \Phi\ind{m}{\mu}(\varphi)
\end{align}
for integers $n, m$ with
\begin{align} \label{eq:indices}
  n\ge0,\quad m\ge0,\quad \nu\in\{-1,1\} ,\quad  \mu\in\{-1,1\} .
\end{align} 
A derivation of the Laplacian equation in toroidal coordinates and the
verification that $\harint{n,m}{\nu,\mu}$ is harmonic can be found in
\cite[p. 434]{Hob1931}.  For the values of $n$ and $m$ specified in
\eqref{eq:indices}, $\harint{n,m}{\nu,\mu}$ is bounded near $S^1$,
which is of measure zero and hence is a removable set for harmonic
functions (cf.\ \cite{AxlerBourdonRamey2001}), so we may write
$\harint{n,m}{\nu,\mu}\in\har(\R^3-\R_0)$ . In particular,
$\harint{n,m}{\nu,\mu}$ is harmonic in $\dom$. Similarly, the exterior
harmonics
\begin{align}\label{eq:exteriorharmonic}
  E\ind{n,m}{\nu,\mu}(x)=E\ind{n,m}{\nu,\mu}[\eta_{0}](x)
  = \chcos \,\frac{P\ind{n-1/2}{m}(\cosh\eta)}{P\ind{n-1/2}{m}(\cosh\eta_{0})} \,
   \Phi\ind{n}{\nu}(\theta) \,  \Phi\ind{m}{\mu}(\varphi)
\end{align}
are in $\har(\R^3-S^1)$.

One may define a weighted  $L^2$ inner product on real-valued functions by
\begin{align}
  \langle f,g \rangle_{\eta_0} =  \iiint_{\dom} f g\, w \,dV
\end{align}
with the weight function
\begin{align}\label{eq:weight}  
   w(\eta,\theta,\varphi) =  \frac{(\cosh\eta-\cos\theta)^2}{\sinh\eta}. 
\end{align}
From this one easily finds the following.

\begin{prop} \label{prop:norms}  The interior
  toroidal harmonics $\{I\ind{n,m}{\nu,\mu}\}$ (for
  $n,m,\nu,\mu$ as in \eqref{eq:indices}) form a complete orthogonal
  system in $L^2(\dom,w)$.  Their norms are 
\begin{align*}
  \| I\ind{n,m}{\nu,\mu}\|^2_{\eta_0} =
  \varepsilon_n\varepsilon_m\int_{n_0}^\infty
  \bigg(\frac{Q\ind{n-1/2}{m}(\cosh\eta)}{Q\ind{n-1/2}{m}(\cosh\eta_{0})}\bigg)^2\,d\eta  
\end{align*}
where $\varepsilon_n= 1+\delta_{n,0}$ and $\delta_{n,m}$ is the
Kronecker delta function. Further, the restrictions to the boundary
$\{\left. I\ind{n,m}{\nu,\mu} \right\vert_{\bdry}\}$ are complete in
$L^2(\bdry)$ and $L^2(\bdry,w)$.
\end{prop}

\begin{proof}
  We only comment on the completeness since the orthogonality is
  trivial. It is well known that
  $\{\Phi\ind{n}{\nu}(\theta) \, \Phi\ind{m}{\mu}(\varphi)\}$ is a
  complete set in $L^2([-\pi, \pi]^2)$, and since the factor $\chcos$
  and the weight function $w$ are bounded from below and above, and
  thus do not affect the completeness \cite[p.\ 154]{Petrovsky1954},
  we have completeness on the boundary. The completeness in the
  interior is similar.
\end{proof}

We will need  the following series expansion.

\begin{prop}[\cite{CD2011}] \label{prop:generalizedheine}
  For all $\alpha\in\C$,
  \begin{align*}
  (\cosh\eta-\cos\theta)^{-\alpha} = \frac{1}{\Gamma(\alpha)}\sqrt{\frac{2}{\pi}}
    \frac{e^{-i\pi(\alpha-1/2)}}{(\sinh\eta)^{\alpha-1/2}}
    \sum_{n-0}^\infty \epsilon_n \cos(n\theta) \, Q_{n-1/2}^{\alpha-1/2}(\cosh\eta) .
  \end{align*}
 \end{prop}

 \section{Dirichlet-to-Neumann
   mapping} \label{Dirichlet-to-Neumann_mapping_Section}

 Given a suitable $f\colon\bdry\to\R$, the Dirichlet-to-Neumann
 mapping is described by finding the harmonic function $u\in\har\dom$
 with boundary values $f=u\vert_\bdry$, and then taking the normal
 derivative $h=\nor u$. The mapping is thus $\Lambda f = h$. A common
 setting \cite{LionsMagenes1972} is for $u$ to be in the Sobolev space
 $H^1(\dom)$ and $f$ in the boundary space $H^{1/2}(\bdry)$.

Roughly speaking, $H^1(\dom)$ consists of $L^2$ functions with $L^2$
derivatives, and $H^{1/2}(\bdry)=H^1(\dom)/H^1_0(\Omega)$ is
identified with the space of boundary values of elements of
$H^1(\dom)$, where $H^1_0(\Omega)$ denotes the closure of the subspace
of functions of compact support. The trace map
$\mbox{tr}\colon H^1(\dom)\to H^{1/2}(\bdry) $ is a bounded linear
function, and for $v\in H^1(\dom)$ we will informally write $v\vert_\bdry$
for $\mbox{tr}[v]$.
 
We have the following result.
%lemma3.1
\begin{lemm}\label{lemm:nori}
  The normal derivative of the interior toroidal harmonics is given by
  the formula 
\begin{align*}  
  \nor& I  \ind{n,m}{\nu,\mu}  = \Bigg(\Phi_{n-1}^{\nu}(\theta) \bigg(
    (1+2n)\cosh\eta_{0} -\Big(2(n-m)+1\Big)
      \frac{Q_{n+1/2}^{m}(\cosh\eta_{0})}{Q_{n-1/2}^{m}(\cosh\eta_{0})}\bigg)  \\
  &  \; +\Phi_{n}^{\nu}(\theta)\bigg(
  -2\Big(  (2n\cosh^{2}\eta_{0}+1) -\Big(2(n-m)+1\Big)\cosh\eta_{0}
  \frac{Q_{n+1/2}^{m}(\cosh\eta_{0})}{Q_{n-1/2}^{m}(\cosh\eta_{0})}\Big)\bigg)\\
  & \;  +\Phi_{n+1}^{\nu}(\theta)\bigg((1+2n)\cosh\eta_{0}  -\Big(2(n-m)+1\Big)
  \frac{Q_{n+1/2}^{m}(\cosh\eta_{0})}{Q_{n-1/2}^{m}(\cosh\eta_{0})}\bigg)\Bigg)  \times\\
  &  \quad \frac{(\cosh\eta_{0}-\cos\theta)^{1/2}}{4\sinh\eta_{0}}  \Phi_{m}^{\mu}(\varphi)  .
\end{align*}
\end{lemm}

 \begin{proof}
  The proof is a straightforward, but tedious calculation based on
  \eqref{eq:unitnormal}, \eqref{eq:norfddeta}, and
  \eqref{eq:interiorharmonic}, and the recurrence formula
  \cite[pp. 161--162]{Bateman1944}
\begin{align*}
  \qquad \sinh^2\!\eta \, (Q^m_{n+1})^{\prime}(\cosh\eta)
  = (n+m+1) Q^m_{n}(\cosh\eta) - (n+1)\cosh\eta \, Q^m_{n+1}(\cosh\eta).
  \qquad  \qquad %\qedhere
\end{align*}
\end{proof}

 Consider the boundary function $f$ represented as
\begin{align} \label{eq:fseries}
 \frac{1}{\chcosz}  f(\theta,\varphi) = \sum_{n,m,\nu,\mu}
  a\ind{n,m}{\nu,\mu} 
 \Phi_n^\nu(\theta) \Phi_m^\mu(\phi),
\end{align} i.e.,
\begin{align} \label{eq:fseries}
  f(\theta,\varphi) = 
  \sum_{n,m,\nu,\mu} a\ind{n,m}{\nu,\mu}
   I\ind{n,m}{\nu,\mu}(\eta_0,\theta,\varphi) 
\end{align} 
as per Proposition \ref{prop:norms}.  Taking into account that
$\Phi\ind{0}{-}=0$ identically, unless otherwise specified, the
indices of summation will always be as in \eqref{eq:indices} but
excluding the cases of $(n,m,\nu,\mu)$ being $(0,m,-1,\mu)$ or
$(n,0,\nu,-1)$. For convenience, we will often write superscripts as
``$+$'' in place of $1$ and ``$-$'' in place of $-1$.

Thus $f=u\vert_\bdry$ where $u\in\har(\dom)$ is given by
\begin{align}
\label{eq:usum}
u=\sum_{n,m,\nu,\mu} a\ind{n,m}{\nu,\mu}I\ind{n,m}{\nu,\mu} .
\end{align}
Then by Lemma \ref{lemm:nori}, we have  $h = \Lambda f$ is in turn given by
\begin{align} \label{eq:hseries}
  h =  \sum_{n,m,\nu,\mu}
a\ind{n,m}{\nu,\mu} 
  \nor I\ind{n,m}{\nu,\mu} ,
\end{align}
assuming that $f$ is sufficiently well-behaved to justify the exchange
of summation and differentiation. For example, since the trace
operator and $\Lambda\colon H^{1/2}(\bdry)\to H^{-1/2}(\bdry)$ are
continuous \cite{salo2012},
\begin{align*}
  \nor u &=  \nor \sum _{n,m,\nu,\mu}a\ind{n,m}{\nu,\mu}I\ind{n,m}{\nu,\mu}
           = \Lambda \big( (\sum _{n,m,\nu,\mu}
           a\ind{n,m}{\nu,\mu}I\ind{n,m}{\nu,\mu})\bigg\vert_\bdry   \big) \\
         &= \Lambda  \big( \sum _{n,m,\nu,\mu}
           a\ind{n,m}{\nu,\mu} (I\ind{n,m}{\nu,\mu}\bigg\vert_\bdry ) \big) \\
         &= \sum _{n,m,\nu,\mu} a\ind{n,m}{\nu,\mu}
           \Lambda ( I\ind{n,m}{\nu,\mu}\bigg\vert_\bdry )  
     = \sum _{n,m,\nu,\mu} a\ind{n,m}{\nu,\mu}\nor I\ind{n,m}{\nu,\mu},
\end{align*}
with the last sum converging in the dual space $H^{-1/2}(\bdry)$ of
$H^{1/2}(\bdry)$.  It is also valid under the assumption that the sum
in \eqref{eq:usum} and the sums of the partial derivatives of the
terms converge uniformly on compact subsets of $\dom$.

%The constant $q_{n,m}$ will simplify the formulas later.

We will show that $\nor u$ can be expressed in terms of the
coefficients of $f$ and certain constants defined in terms of Legendre
functions. We will use the abbreviations
\begin{align}  \label{eq:abbrev}
   t_0=\cosh \eta_0,\quad  s_0=\sinh \eta_0,\quad q_{n,m} = Q_{n-1/2}^m(t_0),
\end{align}
and will make use of the following constants.
%Def3.2
\begin{defi} \label{def:ncoefs}
  The \textit{toroidal Neumann constants},
  $\rho_{n,m} = \rho_{n,m}(\eta_{0})$, 
  $\sigma_{n,m} = \sigma_{n,m}(\eta_{0})$, and
  $\tau_{n,m}= \tau_{n,m}(\eta_{0})$, 
 are defined as follows:
  \begin{align}
    \rho_{1,m} &= \frac{1}{2s_0}\big(t_0 + (2m-1)\frac{q_{1,m}}{q_{0,m}}\big)  , \nonumber\\ 
    \rho_{n,m} &= \frac{1}{4s_0}\Big((2n-1) t_0+ (2(m-n)+1) \frac{q_{n,m}}{q_{n-1,m}}
                 \Big) \ \ (n\ge2), \nonumber \\%\label{eq:defkappa}\\
   % \sigma_{0,m}  &=  \frac{-1}{2s_0}(1+(2m-1)t_0\frac{q_{1,m}}{q_{0,m}})  , \nonumber\\
    \sigma_{n,m}
    &=  \frac{-1}{2s_0}\Big((2nt_{0}^{2}+1)  +  (2(m-n)-1)t_0 \frac{q_{n+1,m}}{q_{n,m}}\Big)  \ \ (n\ge0) ,\nonumber \\% \label{eq:deflambda} \\
 %   \tau_{0,m}&=  \frac{1}{4s_0}\big(3 t_0  + (2m-3)\frac{q_{2,m}}{q_{1,m}}\big) ,\nonumber\\
    \tau_{n,m}&=  \frac{1}{4s_0}\Big((2n+3) t_0 +
 (2(m-n)-3) \frac{q_{n+2, m}}{q_{n+1,m}}\Big)  \ \ (n\ge0), \label{eq:defrhosigmatau} 
\end{align}
for all $m\ge0$. 

\end{defi}

We will need the following asymptotic values.  In \cite[p.\
305]{Hob1931} it is shown that for fixed $\eta_0$ and $m$,
\begin{align} \label{eq:qasymptotic}
   Q_n^m(\cosh \eta_0)  \ \sim 
  (-1)^m \frac{\Gamma[n+m+1]}{\Gamma[n+1]} \big(\frac{\pi}{n}\big)^{1/2}
  \frac{e^{-(n+1/2)\eta_0}}{(2\sinh\eta_0)^{1/2}} 
\end{align}
as $n\to\infty$, where $\sim$ means that the ratio of the two
expressions tends to 1.  From this it is seen that
\begin{align} \label{eq:qratiolimit}
   \lim_{n\to\infty}  \frac{ q_{n,m} }{ q_{n-1,m} } = e^{-\eta_0}   
\end{align}
independently of the value of $m$. Therefore
\begin{align}  \label{eq:coeflimits}
  \lim_{n\to\infty} \frac{\rho_{n,m}}{n} = \frac{1}{2}, \quad
    \lim_{n\to\infty} \frac{\sigma_{n,m}}{n} = -t_0,  \quad
  \lim_{n\to\infty} \frac{\tau_{n,m}}{n} = \frac{1}{2}.
\end{align}

We will also use the following facts about Fourier coefficients
\cite[Corollary 3.3.10,\ Proposition 3.3.12]{Gr2014}:
\begin{prop} \label{prop:smoothness}
  If $f(\theta,\varphi)$ is of class $C^r$, then 
  \begin{align}    \label{eq:coeffbound}
     \vert a\ind{n,m}{\nu,\mu}\vert \le \frac{C}{(m+n+1)^r}
  \end{align}
  for some constant $C>0$.  Conversely, if \eqref{eq:coeffbound} holds
  and $r\ge2$, then $f$ is of class $C^{r-2}$.
\end{prop}

\begin{theo} \label{theo:bfroma} For a fixed $\eta_0$, let
  $f\colon\bdry\to\R$ given by \eqref{eq:fseries} and suppose that the
  Fourier coefficients satisfy \eqref{eq:coeffbound} where
  $r\ge4$. for some constant $C$.  Define, for $\mu=\pm1$ and all
  $m\ge0$,
\begin{align}
  b\ind{0,m}{+,\mu} &= \sigma_{0,m}\,  a\ind{0,m}{+,\mu}
   + \tau\ind{0,m}\,  a\ind{1,m}{+,\mu}, \nonumber \\
  b\ind{n,m}{+,\mu} &= \rho_{n,m}\, a\ind{n-1,m}{+,\mu}
   + \sigma_{n,m}\,  a\ind{n,m}{+,\mu}
   + \tau_{n,m}\,  a\ind{n+1,m}{+,\mu} \ \ (n\ge1), \nonumber\\[1ex]
  b\ind{1,m}{-,\mu} &= \sigma_{1,m} \, a\ind{1,m}{-,\mu}
   + \tau_{1,m}\,  a\ind{2,m}{-,\mu}, \nonumber \\
  b\ind{n,m}{-,\mu} &= \rho_{n,m}\,  a\ind{n-1,m}{-,\mu}
   + \sigma_{n,m}\,  a\ind{n,m}{-,\mu}
   + \tau_{n,m}\, a\ind{n+1,m}{-,\mu} \ \ (n\ge2). \label{eq:bfroma}
\end{align}
Then the Dirichlet-to-Neumann  mapping $h=\Lambda f$ is given by the formula
\begin{align} 
\label{eq:defh}
  h(\theta,\varphi) = \sqrt{\cosh\eta_{0}-\cos\theta}
  \sum_{n,m,\nu,\mu} b\ind{n,m}{\nu,\mu}
   \Phi\ind{n}{\nu}(\theta) \Phi\ind{m}{\mu}(\varphi)    ,
\end{align}
which converges absolutely and $h$ is of class $r-3$.
\end{theo}

\begin{proof}
  From \eqref{eq:interiorharmonic}
  it follows that
  \begin{align*}
    \big\vert I\ind{n,m}{\nu,\mu}(\eta_0,\theta,\varphi) \big\vert \le  
    \sqrt{t_0+1}
  \end{align*}
  so the expansion
  \eqref{eq:fseries} converges absolutely.  
  Similarly,
  one verifies from Lemma \ref{lemm:nori}  and  \eqref{eq:qratiolimit}   that 
  \begin{align*}
    \big\vert \nor I\ind{n,m}{\nu,\mu} \big\vert \le C_1( n+m+1)
  \end{align*}
  for some constant $C_1$ (which depends on $\eta_0$), and hence by
  \eqref{eq:coeffbound}
    \begin{align*}
      \big\vert a\ind{n,m}{\nu,\mu}\nor I\ind{n,m}{\nu,\mu} \big\vert \le
      \frac{CC_1}{(n+m+1)^{r-1}}.
  \end{align*}
  This is enough to guarantee that \eqref{eq:hseries}, a double series
  in $\theta$ and $\varphi$, also converges absolutely. (Recall however that
  $\sum 1/(m+n+1)^2=\infty$.) This in turn
  permits us to substitute the formula for $\nor I\ind{n,m}{\nu,\mu}$
  into \eqref{eq:hseries} and then reindex $\Phi_{n-1}^{\nu}(\theta)$
  and $\Phi_{n+1}^{\nu}(\theta)$ into $\Phi_n^{\nu}(\theta)$ to obtain
  after some calculation that
  \begin{align} \label{eq:hD2Nseries}
    h(\theta,\varphi) = 
    \sqrt{t_{0}-\cos\theta}\sum_{m,n,\nu,\mu}
    \Phi_{n}^{\nu}(\theta)\Phi_{m}^{\mu}(\varphi)\big(\rho_{n,m}a_{n-1,m}^{\nu,\mu}+
    \sigma_{n,m}a_{n,m}^{\nu,\mu}+\tau_{n,m}a_{n+1,m}^{\nu,\mu}\big) ,
\end{align}
which is \eqref{eq:defh}. By \eqref{eq:coeflimits}, the Fourier
coefficients $b\ind{n,m}{\nu,\mu}$ defining $h$ are of order no
greater than $1/((m+n+1)^ {r-1})$, so by Proposition
\ref{prop:smoothness} we are done.
\end{proof}

Since we are mainly interested in the numerical relationships, we will not
go deeper into relaxing the condition on the coefficients.

\section{Neumann problem} \label{Neumann_problem}

The Dirichlet problem, that is, to find a harmonic function $u$ with
boundary values $f$, is conceptually simple when expressed in terms of
a basis of harmonic functions and was, in fact, implicitly solved for
the torus in the course of construction of the Dirichlet-to-Neumann
mapping which we gave above.  However, the Neumann problem, which
consists of finding a boundary function $f$ with a prescribed normal
derivative, presents special challenges since it is of the nature of
an inverse operation.

\subsection{Algebraic solutions for the Neumann coefficients}

The solution of the Neumann problem, in general, is guaranteed by the
following result \cite{DK1987,Mikh1978,Folland1995}, valid for domains
$\Omega$ with sufficiently smooth boundary.
\begin{prop} \label{prop:neumanntheorem}
  Let $h\in H^{-1/2}(\partial\Omega)$ satisfy the compatibility condition
  \begin{align}
  \label{eq:compat}
    \int_{\partial\Omega} h\, dS = 0.
  \end{align}
  Then there exists an $f\in H^{1/2}(\partial\Omega)$ such that
  $\Lambda f = h$.  This solution is unique up to an additive
  constant. If $h\in L^2(\bdry)$, then $f\in L^2(\bdry)$. If $h$ is
  continuous, then $f$ is continuous.
\end{prop}

The solution $f$ can made unique by applying the \textit{normalization
  condition}
\begin{align} \label{eq:normalization}
   \int_\bdry f\, dS = c.
\end{align}
for a chosen constant c.

\begin{lemm}
  (a) The compatibility condition \eqref{eq:defh} applied to $h$ of the
  form \eqref{eq:hseries} is equivalent to 
\begin{align} \label{eq:bcompat}
  \sum_{n=0}^\infty \varepsilon_n^2\, Q_{n-1/2}^1(\cosh\eta_0)\, b\ind{n,0}{+,+} = 0.
\end{align}
(b) The normalization condition \eqref{eq:normalization} applied to
$f$ of the form \eqref{eq:fseries} is equivalent to
\begin{align} \label{eq:acompat}
  \sum_{n=0}^\infty \varepsilon_n^2\, Q_{n-1/2}^1(\cosh\eta_0)\, a\ind{n,0}{+,+}
  = -\frac{c}{4\pi\sqrt{2} }.
\end{align}
\end{lemm}

\begin{proof}
  (a)  Since $\int_{0}^{2\pi}\Phi_n^-(\theta)\,d\theta=0$, while
  $\int_{0}^{2\pi}\Phi_0^+ (\varphi)\, d\varphi=2\pi$,
  $\int_{0}^{2\pi}\Phi_m^+ (\varphi)\, d \varphi=0$ for $m\ge1$, 
\begin{align*}
   \iint _{\partial\Omega_{\eta_{0}}}h (\theta,\varphi) \, dS&=
   \int_{0}^{2\pi} \int_{0}^{2\pi}h(\theta,\varphi) \, \frac{s_{0}}{(t_{0}-\cos\theta)^{2}} \, d\theta d\varphi\\&
   = s_{0}\sum_{m}\sum_{n}b_{n,m}^{\nu,\mu}\int_{0}^{2\pi}(t_{0}-\cos\theta)^{-3/2}\Phi_{n}^{\nu}(\theta) \, d\theta\int_{0}^{2\pi}\Phi_{m}^{\mu}(\varphi) \, d\varphi
   \\&
   =-4\sqrt{2}\,\pi\sum_{n=0}^{\infty}\varepsilon_{n}^{2}b_{n,0}^{+,+}Q_{n-\frac{1}{2}}^{1}(\cosh\eta_{0}),
\end{align*}
with the last equality following from \eqref{prop:generalizedheine}
with $\alpha=3/2$. The proof of (b) follows the same lines as (a).
\end{proof}

\begin{defi}
  Consider a Neumann problem defined by a function $h\colon\bdry\to\R$
  with an expansion \eqref{eq:defh}.  We will say that a collection of
  real numbers $\{a\ind{n,m}{\nu,\mu}\}$ is an \textit{algebraic
    solution} of the Neumann problem when all of the equations
  \eqref{eq:bfroma} are satisfied.
\end{defi}
 
In order to generate a solution to the Neumann problem, the algebraic
solution must, in fact, provide a convergent series in
\eqref{eq:fseries}.  It is clear that the subcollection of equations
\eqref{eq:bfroma} determined by fixed values of $m$, $\nu$, and $\mu$
are independent of the equations determined by other values of these
parameters.

\begin{lemm} \label{lemm:nonzerocoef} The values $\tau_{n,m}$ are never zero.
\end{lemm}

\begin{proof}
In \cite[p.\ 195]{Hob1931} it is shown that
\begin{align*} %   \label{eq:hobsonq}
  Q_n^m(t) = \frac{(-1)^m}{2^{n + 1}} \frac{(n + m)!}{n!} \, (t^2-1)^{m/2}
  \int_{-1}^1   \frac{(1 - s^2)^n}{(t - s)^{n+m+1}} \,ds   
\end{align*}
(in fact, this rather than \eqref{eq:legendrefunctions} is taken as
the definition of $Q_n^m(t)$ for $n,m\in\C$).  From this it follows
that
\begin{align}  \label{eq:qnonzero}
  (-1)^m Q_n^m(\cosh \eta) > 0
\end{align}
for all $n,m,\eta\in\R^+$.  From \eqref{eq:defrhosigmatau}, we need to
show that the value
\begin{align}
\label{eq:tobenonzero}
   16 s_{0}\,q_{n+1,m}\,\tau_{n,m}={4}\big((2n+3)t_{0}\,q_{n+1,m}+\big(2(n-m)-3\big)q_{n+2,m}\big)
\end{align}
does not vanish.  Consider the recursion formulas from \cite[pp.\
161--162]{Bateman1944} and \cite[p.\ 108]{Hob1931}
\begin{align*}  
    (n-m+1)Q_{n+1}^{m}(t) &= (2n+1)\,t\, Q_{n}^{m}(t)-(n+m)Q_{n-1}^{m}(t), \\
  (t^{2}-1)^{1/2} Q_{n+1}^{m}(t) &= \frac{1}{2n+3}\big(Q_{n+2}^{m+1}(t)-Q_{n}^{m+1}(t)\big); 
\end{align*}
i.e.,
\begin{align}
  2(n+1)\,s_{0}\,q_{n+1,m-1}&=q_{n,m}-q_{n+2,m} ,  \label{eq:recursionapply1}\\
  2(n+1)\,t_{0}\,q_{n+1,m} &=(n-m+\frac{3}{2})q_{n+2,m}+(n+m+\frac{1}{2})q_{n,m} .\label{eq:recursionapply2}
\end{align}
Applying \eqref{eq:recursionapply1} to \eqref{eq:tobenonzero}, we find
\begin{align*}
 16(n+1)\,s_{0}\,q_{n+1,m}\,\tau_{n,m}=(2n+3)\big(2(m+n)+1\big)q_{n,m}-(1+2n)\big(2n-2m+3\big)q_{n+2,m}.
\end{align*}
Now add and subtract $(2n+3)(2(m+n)+1)q_{n+2,m}$ and use
\eqref{eq:recursionapply2}, yielding
    \begin{align*}
16\,s_{0}\,q_{n+1,m}\,\tau_{n,m}=-2(2n+3)(2n+2m+1)q_{n+1,m-1}+8\,m\,q_{n+2,m},
\end{align*}
 which by (\ref{eq:qnonzero})  is never zero.
\end{proof}

%%%%%%%%%%%%%%%%%%%added here%%%%%%%%%%%%%%%%%%%%%%%%%%%%%% 

By Lemma \ref{lemm:nonzerocoef}, when $\mu$ and $m$ are specified,
using arbitrary values of $a\ind{0,m}{+,\mu}$ or $a\ind{1,m}{-,\mu}$,
one may solve the first equations of \eqref{eq:bfroma} to find
\begin{align*}
  a\ind{1,m}{+,\mu}  &= \frac{1}{\tau_{0,m}}
      (b\ind{0,m}{+,\mu} - \sigma_{0,m}\,  a\ind{0,m}{+,\mu} ) ,
\end{align*}
or
\begin{align*}
  a\ind{2,m}{-,\mu}  &= \frac{1}{\tau_{1,m}}
      (b\ind{1,m}{-,\mu} - \sigma_{1,m}\, a\ind{1,m}{-,\mu} )
\end{align*}
respectively. Then the remaining equations may be solved
successively. If this is done for all admissible combinations of
$(m,\nu,\mu)$, an algebraic solution for \eqref{eq:bfroma} is
obtained, uniquely determined by the collection of initial values
$\{a\ind{0,m}{+,\mu}, \ a\ind{1,m}{-,\mu}\}$.

\begin{remark}
  When one applies the strategy given above to the Neumann problem on
  the sphere $|x|<1$ in spherical coordinates $y_0=\rho\cos\theta$,
  $y_1=\rho\sin\theta\cos\phi$, $y_2=\rho\sin\theta\sin\phi$, and uses
  the standard solid spherical harmonics
  $Y_{n,m}^\pm=\rho^n P_n^m(\cos\theta)\Phi_m^\pm(\varphi)$,
  $0\le m\le n$ as the basis for the harmonic functions, it is natural
  to represent the Dirichlet and Neumann functions $f(\theta,\varphi)$
  and $h(\theta,\varphi)$ with the well-known basis
  $\{Y_{n,m}^\pm(1,\theta,\varphi)$\} for $L^2$ functions on the
  sphere \cite{Raynor1934}. This yields certain coefficients
  $a_{n,m}^\pm$ and $b_{n,m}^\pm$, respectively.  Since the normal
  derivative in this situation is equal to the radial derivative,
  $\nor u = (\partial u/\partial\rho)\vert_{\rho=1}$, one sees immediately
  that $b_{n,m}^\pm=na_{n,m}^\pm$, so the analogue of the system
  \eqref{eq:bfroma} is rather trivial. We have not seen this type of
  solution presented in the literature. (In \cite[p.\
  218]{Lebedev1965}, this approach is suggested in a remark after
  expressing the solution to the Dirichlet problem for a spheroid, but
  only for functions constant with respect to the angular coordinate,
  i.e., involving $P_n$ but not general $P_n^m$.)
\end{remark}

\subsection{When does an algebraic solution give a convergent series?\label{subsec:convergent}}

Now we investigate how to obtain algebraic solutions which, in fact,
give solutions to the Neumann problem. It is not difficult to verify
that when $h$ in Proposition   \ref{prop:neumanntheorem} is real analytic,
the solution $f$ to the Neumann problem is also real analytic.

We have the following. Write
$f\ind{m_0}{\nu_0,\mu_0}$ for the sum over $n$ of those terms of the
series \eqref{eq:fseries} for which $(m,\nu,\mu)=(m_0,\nu_0,\mu_0)$.

\begin{prop} \label{prop:optimala} Let the coefficients
  $\{b\ind{n,m}{\nu,\mu}\}$ be such that  the series
  \eqref{eq:defh} converges absolutely, defining
  $h\colon\bdry\to\R$. Assume that $b\ind{n,0}{+,+}$ satisfy
  \eqref{eq:bcompat}, so $h$ satisfies the compatibility condition
  \eqref{eq:compat}. Suppose further that the continuous solution
  $f\colon\bdry\to\R$ of $\Lambda f=h$ specified in Proposition
  \ref{prop:neumanntheorem} has a double Fourier series which
  converges absolutely. Then (i) for every value of
  $a\ind{0,0}{+,+}\in\R$, the resulting algebraic solution for the
  sequence $\{a\ind{n,0}{+,+}\}$ produces an absolutely convergent
  series $\sum_{n} a\ind{n,0}{+,+}I\ind{n,0}{+,+}(\eta_0,\theta,\phi)$
  whose value is $f\ind{0}{+,+}$ plus a constant.  Further, (ii) for
  $(m,\nu,\mu)$ different from $(0,+1,+1)$, there exists a unique
  value of $a\ind{0,m}{\nu,\mu}$ (when $\nu=1$) or
  $a\ind{1,m}{\nu,\mu}$ (when $\nu=-1$) for which the resulting
  algebraic solution gives a convergent series
  $\sum_{n}
  a\ind{n,m}{\nu,\mu}I\ind{n,m}{\nu,\mu}(\eta_0,\theta,\phi)$. The sum
  of this series is $f\ind{m}{\nu,\mu}$.
  
\end{prop}

Given the convergence criteria of Theorem \ref{theo:bfroma}, almost
all of Proposition \ref{prop:optimala} follows immediately from
Proposition \ref{prop:neumanntheorem} (together with the observation
that the solutions of the subsystem for each combination of
$(m,\nu,\mu)$ are essentially independent).

\subsection{Convergence for the indices  $(m,\nu,\mu)=(0,+1,+1)$\label{subsec:0++}}

The only assertion of Proposition \ref{prop:optimala} which remains to
be verified is that the value of the coefficient $a\ind{0,0}{+,+}$
referred to in part (i)  is arbitrary. First,
we observe that for the particular indices
$(m,\nu,\mu)=(0,+1,+1)$, the Neumann constants satisfy some special
relations.

\begin{lemm} \label{lemm:0relations}
  $\sigma_{0,0}\, q_{0,0} + 2\,\tau_{0,0}\, q_{1,0} =0$;
  $\rho_{1,0}\, q_{0,0} + 2(\sigma_{1,0} \, q_{1,0}+ \tau_{1,0}\, q_{2,0})=0$;
  and for $n\ge2$, 
  \[ \rho_{n,0}\, q_{n-1,0}+ \sigma_{n,0} \, q_{n,0}+ \tau_{n,0} \, q_{n+1,0}= 0.
  \]
\end{lemm}

\begin{proof}  
  Via the recursion formula (\ref{eq:recursionapply2}) as well as the
  following \cite[pp.\ 161--162]{Bateman1944},
\begin{align*}
  (n-m)t\, Q_{n}^{m}(t) = (t^2-1)^{1/2} \, Q_{n}^{m+1}(t)+(n+m)Q_{n-1}^{m}(t),
\end{align*}
i.e.,
\begin{align}  \label{eq:recursionapply3}
 (n-m-\frac{1}{2})t_{0}q_{n,m}=s_{0}\,q_{n,m+1}+(n+m-\frac{1}{2})\,q_{n-1,m}
\end{align}
direct computations show that
\begin{align*}
  \sigma_{0,0}\, q_{0,0} + 2\,\tau_{0,0}\, q_{1,0} =& \, \frac{-1}{2s_{0}}\big(q_{0,0}-t_{0} q_{1,0} \big)+\frac{3}{2s_{0}}\big(t_{0}q_{1,0}-q_{2,0}\big)\\
  =&\, \frac{3}{8s_{0}}\big(q_{2,0}-q_{0,0}\big)-\frac{3}{8s_{0}}\big(q_{2,0}-q_{0,0}\big).
\end{align*}
Similarly, we find
   \begin{align*}
\rho_{1,0} &\, q_{0,0} + 2(\sigma_{1,0} \, q_{1,0}+ \tau_{1,0}\, q_{2,0})
     \\=& \, \frac{1}{4s_{0}}\big(t_{0}\,q_{0,0}-q_{1,0}\big)+
          \frac{1}{s_{0}} \big(-(2t_{0}^{2}+1)\,q_{1,0}+
          \frac{11}{2}\,t_{0}\,q_{2,0}-\frac{5}{2}\,q_{3,0}\big)\\=&
   -\frac{1}{2}\,q_{0,0}+\frac{1}{s_{0}}\big(\frac{1}{2}q_{1,0}+\frac{1}{2}\,s_{0}\,q_{0,1}-\frac{1}{2}\,q_{-1,0}\big )\\=& \, 0.
   \end{align*} 
Finally,
   \begin{align*}
    \rho_{n,0}\, q_{n-1,0}+ \sigma_{n,0} \, q_{n,0}+ \tau_{n,0} \, q_{n+1,0}  &=  (2n-1)t_{0} q_{n-1,0} +(-2n+1-4nt_{0}^{2}-2)q_{n,0} \\
   & \quad\  +(6n+5)t_{0}q_{n+1,0}-(2n+3)q_{n+2,0}\\ & =0.  %\qedhere
      \end{align*}      
\end{proof}

When $b\ind{n,0}{+,+}=0$ for all $n$, the corresponding equations
\eqref{eq:bfroma} are linear homogeneous, and Lemma
\ref{lemm:0relations} implies that
\begin{align}  \label{eq:constsol}
 \frac{a\ind{n,0}{+,+}}{q_{n,0}} = 2 \frac{a\ind{0,0}{+,+}}{q_{0,0}}\ \ (n\ge 1);
\end{align}
i.e., $a\ind{n,0}{+,+} = \varepsilon_n a\ind{0,0}{+,+}$ for all $n$.
On comparing the formula of Proposition
\ref{prop:generalizedheine} with the exponent determined by
$\alpha=1/2$, one sees that the solution
\begin{align*}
  f_{0}^{+,+}(\theta,\varphi) = \frac {a\ind{0,0}{+,+}}{q_{0,0}}
  \sum_{n=0}^\infty \varepsilon_n q_{n,0}\cos n\theta\,\sqrt{t_0-\cos\theta} 
\end{align*}
is indeed equal to the constant function on
$\bdry$ with value $(\pi/\sqrt{2})(a\ind{0,0}{+,+}/q_{0,0})$. The
solution to the Dirichlet problem in $\dom$ is the same
constant. Thus, given any $a\ind{0,0}{+,+}\in \mathbb{R}$, the
algebraic solution gives a convergent series
$\sum_{n=0}^\infty
a\ind{n,0}{+,+}I\ind{n,0}{+,+}(\eta_0,\theta,\varphi)$, and the initial
value $ a\ind{0,0}{+,+}+c$ in the system \eqref{eq:bfroma} will
generate the series
$\sum_{n=0}^\infty (a\ind{n,0}{+,+}+c\varepsilon_n)
I\ind{n,0}{+,+}(\eta_0,\theta,\varphi)$ which also converges.  This
confirms the nonuniqueness statement we made immediately after
Proposition \ref{prop:neumanntheorem}.

Now suppose for a moment that $h$ is identically zero, so always
$b\ind{n,m}{\pm,\pm}=0$.  For indices with $(m,\nu,\mu)\not=(0,1,1)$,
the unique value of $a\ind{0,m}{\nu,\mu}$ provided by Proposition
\ref{prop:optimala} for generating a convergent series is clearly
$a\ind{0,m}{\nu,\mu}=0$ (with more work, one could also see this by
solving the linear homogeneous system explicitly for a nonzero
starting value and verifying via properties of Legendre functions that
the result does not converge). Returning to arbitrary $h$ satisfying
the compatibility condition, we see that starting from any solution
$f$ given by Proposition \ref{prop:optimala}, we may add any multiple
of the sequence $\epsilon_n$ to the coefficients $a\ind{0,0}{+,+}$,
leaving the remaining $a\ind{n,m}{\nu,\mu}$ unchanged, and obtain
another algebraic solution which, in fact, converges. This verifies
the above statement that arbitrarily chosen $a\ind{0,0}{+,+}$ will
produce an algebraic solution which defines a convergent series. These
considerations also lead to the following.

\begin{prop}
  The area of $\bdry$ is equal to
  \begin{align*}
    \alpha(\eta_0) = -8 \sum_{n=0}^\infty \varepsilon_n^3 q_{n,0}q_{n,1}.
  \end{align*}
\end{prop}

\begin{proof}
  Take $a\ind{0,0}{+,+}=(\sqrt{2}/\pi)q_{0,0}$, which gives $f\ind{0}{+,+}=1$
  identically. Then apply \eqref{eq:acompat} to evaluate
  $\alpha(\eta_0)=\int_\bdry f\ind{0}{+,+}\,dS$.
\end{proof}

\begin{coro}
  Let $f$ be a particular solution of $\Lambda f=h$ and set
  $c_1=\int_\bdry f\,dS$.  Let $\hat f$ be obtained by replacing the
  coefficients $a\ind{n,0}{+,+}$ for $f$ with
    \begin{align*}
      \hat a\ind{n,0}{+,+} =  a\ind{n,0}{+,+} +
      \varepsilon_n \frac{\sqrt{2}}{\pi} \frac{q_{n,1}}{\alpha(\eta_0)}(c-c_1).
    \end{align*}
    Then $\hat f$ is the unique solution of the Neumann problem which
    satisfies the normalization condition \eqref{eq:normalization}.
\end{coro}
  
\subsection{Determination of parameter for convergence    for other values of  $(m,\nu,\mu)$   \label{subsec:msigtau}}

We assume now that $(m,\nu,\mu)\not=(0,+1,+1)$. The essence of the
matter is that the linear system \eqref{eq:bfroma} will only have a
unique solution after one of the variables is arbitrarily chosen, let
us say $a\ind{0,m}{+,\mu}$ (or $a\ind{1,m}{-,\mu}$).  For simplicity
of notation, we will write $a_n$ and $I_n$ in place of
$a\ind{n,m}{\nu,\mu}$ and $I\ind{n,m}{\nu,\mu}$. We will assume that
$\mu=1$ since the case $\mu=-1$ is analogous, the only difference
being the start of the indexing from $n=1$ instead of $n=0$.

Given $a\in\R$, write $A_n(a)$ for the value of $a_n$ in the solution
of the corresponding equations \eqref{eq:bfroma} determined by setting
the arbitrary parameter $a_0 = a\ind{0,m}{+,\mu}$ (or
$a\ind{1,m}{-,\mu}$) equal to $a$. Thus $A_0(a)=a$, and by a simple
induction, we have recursively defined linear expressions
\begin{align}  \label{eq:linearpoly}
  A_n(a) &= C_n a + D_n  \ (n\ge0),
\end{align}
where
\begin{alignat}{2}  \label{eq:recursive}
        C_0 &= 1,  &D_0&=0 , \nonumber\\
        C_1 &=  \frac{-\sigma_0}{\tau_0}, &D_1 &= \frac{b_0}{\tau_0} ,\nonumber\\
        C_{n+1} &= \frac{-1}{\tau_n}(\rho_ nC_{n-1} + \sigma_nC_n), \quad       
        &D_{n+1} &= \frac{-1}{\tau_n}(\rho_nD_{n-1} + \sigma_nD_n - b_n) \ \ (n\ge1).
\end{alignat}
By construction, the collection $\{A_n(a)\}$ is an algebraic solution of the
system \eqref{eq:bfroma}, whatever the value of $a$ may be.
According to \eqref{eq:fseries} and Theorem \ref{theo:bfroma}, we want to
find the unique value $a_\opt$ provided by Proposition
\ref{prop:optimala} for which
\begin{align}  \label{eq:aconvergent} 
  \sum_{n=0}^\infty A_n(a_\opt) \, \Phi\ind{n}{\nu}(\theta)\Phi\ind{m}{\mu}(\varphi)   
\end{align}
converges absolutely and thus gives
$f_{m}^{\nu,\mu}(\theta,\varphi)$. In particular, it is necessary that
$ A_n(a_\opt)\to0$ as $n\to\infty$. By \eqref{eq:linearpoly}, this
says $C_n a_\opt + D_n \to 0$.

Note that the $C_n$ do not depend on the data $\{b_n\}$. It is clear that
two consecutive terms $C_n$, $C_{n+1}$ can never vanish.  Under the
assumption that  $C_n>\epsilon>0$ for infinitely many $n$, we have
\begin{align} \label{eq:Dn/Cn}
  -\frac{D_n}{C_n} \to a_\opt
\end{align}
as $n\to\infty$ on that subsequence. We will look further into this
question in the next section.

%%%%%%%%%%%%%%%%%%%%%%%%
\subsection{Numerical results}

We illustrate the solution of the Neumann problem with numerical
examples.

%%%%%%%%%%%%%%%%%%%%%%%%%%%%
\textit{Vanishing normal derivative.}
Consider $m=0$, $\nu=\mu=+1$. Recall that for this particular
combination of $(m,\nu,\mu)$, the corresponding algebraic solution gives a
solution to the Neumann problem for every choice of
$a_0=a\ind{0,0}{++}$. To calculate this, one simply takes
$b_n=0$ for all $n$ (i.e., the coefficients of the
vanishing normal derivative described in Subsection
\ref{subsec:convergent}). The formulas \eqref{eq:recursive} give
$D_n=0$ always, and by \eqref{eq:linearpoly}, we have
$A_n(a)=C_n\,a$. Choosing $a_0=1$ without any loss of generality
and fixing $\eta_0$, one obtains the values $C_n$ by
\eqref{eq:recursive} and then the initial coefficients $a_n$ by
\eqref{eq:linearpoly}. This amounts to calculating values  of
the associated Legendre functions of the second kind via the recursion
formulas, and the only numerical error is that which accumulates
due to roundoff.
 
%%%%%%%%%%%%%%%%%%%%%%%%%%%%%%%%%%%%
\example
\textit{Numerical behavior of $C_n$.} We observed that
$a_\opt$ is given by \eqref{eq:Dn/Cn}  unless $C_n\to0$. (Recall that the
$C_n$ do not depend on the Neumann data.) For small values of $n$, we
have little control over even the sign of the coefficients defined in
\eqref{eq:defrhosigmatau}. However, from \eqref{eq:coeflimits},
$\rho_{n,m}/\tau_{n,m}\to1$ and $\sigma_{n,m}/\tau_{n,m}\to -2t_0$.
Therefore if for a single large $n$ we have
\begin{align*}
  C_n \approx e^{\eta_0}C_{n-1},
\end{align*}
then by \eqref{eq:recursive} it would follow that
\begin{align*}
  C_{n+1} \approx -C_{n-1}+2t_0C_n = -e^{-\eta_0}C_n+2t_0C_n = e^{\eta_0}C_n;
\end{align*}
i.e.\ the sequence $\{C_n\}$ grows exponentially.  Table \ref{tab:Cn1}
lists calculated values of $C_n$ corresponding to $m=1$ and a range of
values of $\eta_0$. Other values of $m$ are shown in Table
\ref{tab:Cn2}.  Even though the initial values can decrease, in all
cases that we have examined it appears that $C_n\to\infty$
exponentially as $n\to\infty$.

 \bigskip
%%%%%%%%%%%%%%%%%%%%%%%%%%%
%%%%%%%% tab:Cn1
\begin{table}[!h] \centering
$  \begin{array}{l||r@{.}l|r@{.}l|r@{.}l|r@{.}l|r@{.}l|r@{.}l|}
 &\multicolumn{2}{c|}{\eta=0.1} &\multicolumn{2}{c|}{\eta=0.3} &\multicolumn{2}{c|}{\eta=0.5} &\multicolumn{2}{c|}{\eta=1.} &\multicolumn{2}{c|}{\eta=1.5} &\multicolumn{2}{c|}{\eta=2.}\\\hline
C_{0} &  1&  &  1&  &  1&  &  1&  &  1&  &  1&  \\
C_{1} &  1&943  &  1&697  &  1&420  &  0&866  &  0&523  &  0&316  \\
C_{2} &  1&852  &  1&418  &  1&098  &  0&720  &  0&599  &  0&557  \\
C_{3} &  1&752  &  1&229  &  0&985  &  0&996  &  1&444  &  2&294  \\
C_{4} &  1&654  &  1&120  &  1&016  &  1&761  &  4&302  &  11&303  \\
C_{5} &  1&562  &  1&076  &  1&171  &  3&476  &  14&040  &  60&780  \\
C_{6} &  1&48  &  1&085  &  1&458  &  7&286  &  48&463  &  345&631  \\
C_{7} &  1&407  &  1&144  &  1&911  &  15&885  &  173&925  &  2043&827  \\
C_{8} &  1&344  &  1&249  &  2&595  &  35&635  &  642&402  &  12440&253  \\
C_{9} &  1&290  &  1&404  &  3&611  &  81&710  &  2425&889  &  77424&156  \\
C_{10} &  1&246  &  1&616  &  5&120  &  190&646  &  9323&354  &  490447&458  \\
C_{15} &  1&139  &  4&005  &  34&827  &  15650&902  & \multicolumn{2}{l|}{ 9.298  \times 10^6 } & \multicolumn{2}{l|}{ 5.953  \times 10^9 } \\
C_{20} &  1&189  &  11&881  &  279&376  & \multicolumn{2}{l|}{ 1.521  \times 10^6 } & \multicolumn{2}{l|}{ 1.099  \times 10^{10} } & \multicolumn{2}{l|}{ 8.571  \times 10^{13} } \\
C_{30} &  1&722  &  132&742  &  22884&183  & \multicolumn{2}{l|}{ 1.839  \times 10^{10} } & \multicolumn{2}{l|}{ 1.969  \times 10^{16} } & \multicolumn{2}{l|}{ 2.278  \times 10^{22} } \\
C_{40} &  3&068  &  1751&093  & \multicolumn{2}{l|}{ 2.221  \times 10^6 } & \multicolumn{2}{l|}{ 2.641  \times 10^{14} } & \multicolumn{2}{l|}{ 4.196  \times 10^{22} } & \multicolumn{2}{l|}{ 7.201  \times 10^{30} } \\
C_{50} &  6&053  &  25335&250  & \multicolumn{2}{l|}{ 2.369  \times 10^8 } & \multicolumn{2}{l|}{ 4.173  \times 10^{18} } & \multicolumn{2}{l|}{ 9.835  \times 10^{28} } & \multicolumn{2}{l|}{ 2.505  \times 10^{39} } \\
\end{array}  
$
\caption{Sample values of $C_n$ for $m=1$,
  $(\nu,\mu)=(+,+)$. 200-digit precision was used to avoid underflow
  in the calculations. \label{tab:Cn1}}
\end{table}
%%%%%%%%%%%%%%%%%%%%%%%%%%%

%%%%%%%%%%%%%%%%%%%%%%%%%%% 
%%%%%%%% tab:Cn2
\begin{table}[!t]
  \centering
$\begin{array}{l||r@{.}l|r@{.}l|r@{.}l|r@{.}l|}
 &\multicolumn{2}{c|}{m=2} &\multicolumn{2}{c|}{m=3} \
&\multicolumn{2}{c|}{m=4} &\multicolumn{2}{c|}{m=5}\\\hline
C_{0} &  1&  &  1&  &  1&  &  1&  \\
C_{1} &  1&333  &  2&112  &  2&711  &  3&119  \\
C_{2} &  2&125  &  4&111  &  6&103  &  7&835  \\
C_{3} &  5&013  &  10&308  &  16&235  &  22&035  \\
C_{4} &  13&765  &  28&988  &  46&956  &  65&655  \\

 C_{10} &  16716&741  &  36306&992  &  61217&891  &  89729&231  \\
C_{15} & \multicolumn{2}{l|}{ 1.014  \times 10^7 } & 
\multicolumn{2}{l|}{ 2.211  \times 10^7 } & \multicolumn{2}{l|}{ 
  3.749  \times 10^7 } & \multicolumn{2}{l|}{ 5.533  \times 10^7 } \\    
C_{20} & \multicolumn{2}{l|}{ 7.279  \times 10^9 } &  
\multicolumn{2}{l|}{ 1.590  \times 10^{10} } & \multicolumn{2}{l|}{ 
2.701  \times 10^{10} } & \multicolumn{2}{l|}{ 3.996  \times 10^{10} } \\

C_{30} & \multicolumn{2}{l|}{ 4.801  \times 10^{15} } &  
\multicolumn{2}{l|}{ 1.050  \times 10^{16} } & \multicolumn{2}{l|}{ 
1.786  \times 10^{16} } & \multicolumn{2}{l|}{ 2.648  \times 10^{16} 
} \\
C_{40} & \multicolumn{2}{l|}{ 3.764  \times 10^{21} } & 
\multicolumn{2}{l|}{ 8.233  \times 10^{21} } & \multicolumn{2}{l|}{  
1.402  \times 10^{22} } & \multicolumn{2}{l|}{ 2.080  \times 10^{22} } 
\\
C_{50} & \multicolumn{2}{l|}{ 3.246  \times 10^{27} } &  
\multicolumn{2}{l|}{ 7.102  \times 10^{27} } & \multicolumn{2}{l|}{  
1.209  \times 10^{28} } & \multicolumn{2}{l|}{ 1.795  \times 10^{28} 
} \\
\end{array}  
$
\caption{\small Sample values of $C_n$ for $\eta=0.4$, $(\nu,\mu)=(+,+)$.\label{tab:Cn2}}
\end{table}
%%%%%%%%%%%%%%%%%%%%%%%%%% 

%%%%%%%%%%%%%%%%%%%%%%%%%%%
\example Let
\begin{align}  \label{eq:exampleu}
  u = \left( \frac{\sinh\eta}{\cosh\eta-\cos\theta } \right)^m  \cos m \varphi .
\end{align}
It is readily checked that $u$ is harmonic and
\begin{align}  \label{eq:noruexample}
  \nor u = m   \left( \frac{\sinh\eta_0}{ \cosh\eta_0-\cos\theta }\right)^m
  \left((\cosh\eta_0-\cos\theta)\coth\eta_0 +\sinh\eta_0\right)\cos m \varphi .
\end{align}
(One also would obtain a harmonic function with $\sin m\varphi$ in place of
$\cos m\varphi$ in \eqref{eq:exampleu}.)
By Proposition \ref{prop:generalizedheine}, the coefficients
  in the series for $u$ are  equal to
\begin{align}
  a\ind{n,m}{\nu,\mu} = (-1)^m \frac{\sqrt{2/\pi}}{\Gamma(m + 1/2)} \,
  \varepsilon_n q_{n,m}.
\end{align}
We substitute these coefficients into \eqref{eq:bfroma} to obtain
numerical values for the $b\ind{n,m}{\nu,\mu}$. Then we compare
truncations of the series \eqref{eq:hD2Nseries} with the true values
of $h=\nor u$ according to \eqref{eq:noruexample}.  Figure
\ref{fig:bfromaerror} displays the base-10 logarithm of the absolute
error for different combinations of $m$ and $\eta_0$. As is expected,
the error is reduced when the number of terms in the series increases.
It is also seen that the error increases steadily when larger values of $m$
and $\eta_0$ are used.

%%%%%%%%%%%%%%%%%%%%%%%%%%%
\begin{figure}[!h]
\pic{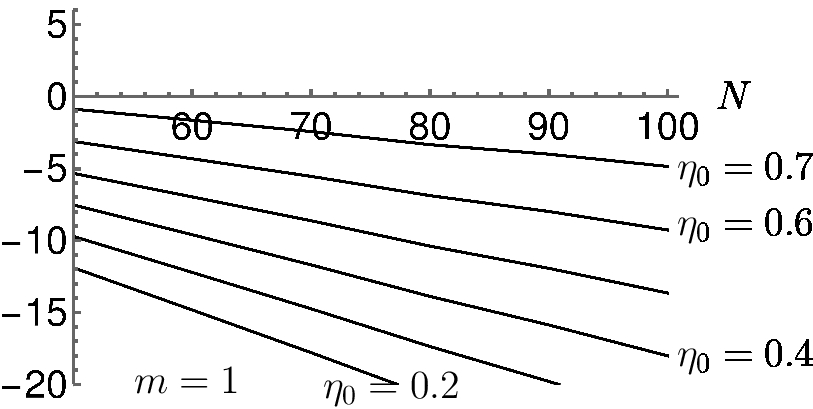}{1.2, -3.5}{1.}{scale=.20} 
\pic{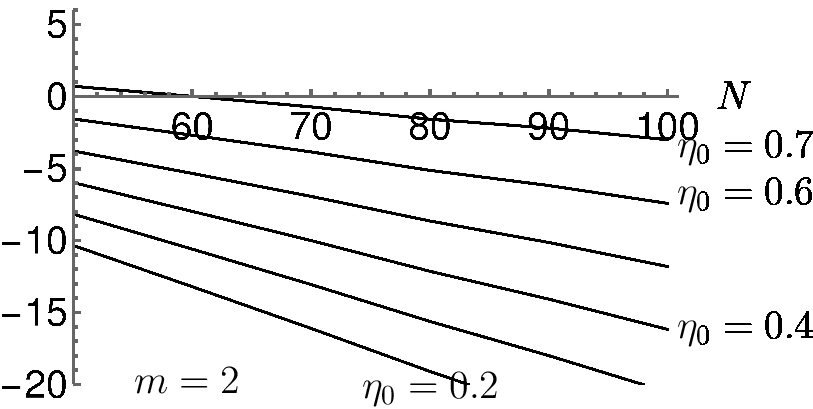}{9.8, -3.5}{1.}{scale=.20} 

\pic{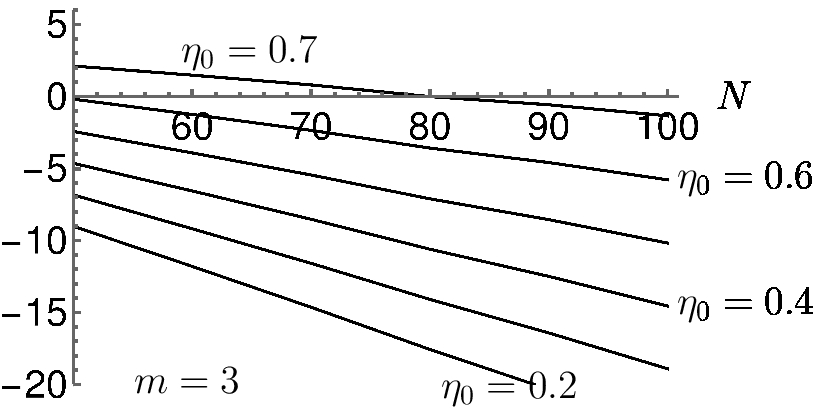}{1.2,-3.5}{1.1}{scale=.20}
\pic{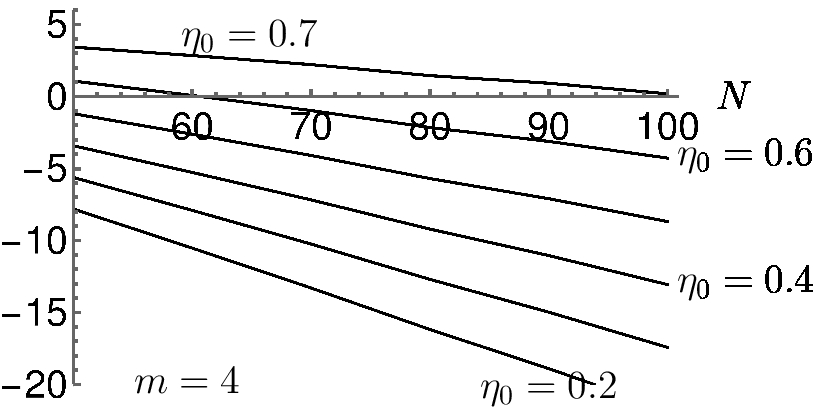}{9.8, -3.5}{1.1}{scale=.20}
\caption{\small Base-10 logarithm indicating number of significant
  figures of approximation of the Dirichlet-to-Neumann mapping given
  by equations \eqref{eq:bfroma} truncating the series
  \eqref{eq:hD2Nseries} to $0\le n\le N$ for varying values of
  $N$. Accuracy is lost as $m$ or $\eta_0$ increases. 100-digit
  precision was used.
  \label{fig:bfromaerror}}
\end{figure}
%%%%%%%%%%%%%%%%%%%%%%%%%%%

%%%%%%%%%%%%%%%%%%%%%%%%%%%%%%%%
\example We now illustrate our algorithm for solving the Neumann
problem. We will use the same function $u$ as in the previous example. The
Fourier coefficients $b\ind{n,m}{\nu,\mu}$ are obtained by numerical
integration. Then the auxiliary coefficients $C_n$, $D_n$ are obtained
recursively by \eqref{eq:recursive}, and then $a_\opt$ is approximated by the last
value of $-C_n/D_n$ according to \eqref{eq:Dn/Cn}.  One would expect
that the values $a\ind{n,m}{\nu,\mu}=A_n(a_\opt)$ of \eqref{eq:Dn/Cn}
provide a convergent series, while for $a\not=a_\opt$, $\{A_n(a)\}$
would not.  This is confirmed by Figure \ref{fig:aopt}, which shows
the values of $A_n(a+\epsilon)$ for small values of $\epsilon$.  The
error in a particular series solution $h$ of the Neumann problem
compared to \eqref{eq:noruexample} is shown in Figure
\ref{fig:neumannerr}.  Maximum errors for combinations of $\eta_0$,
$m$ are shown in Table\ \ref{tab:neumannerror}.

%%%%%%%%%%%%%%%%%%%%%%%%%%%
\begin{figure}[b!]
  \pic{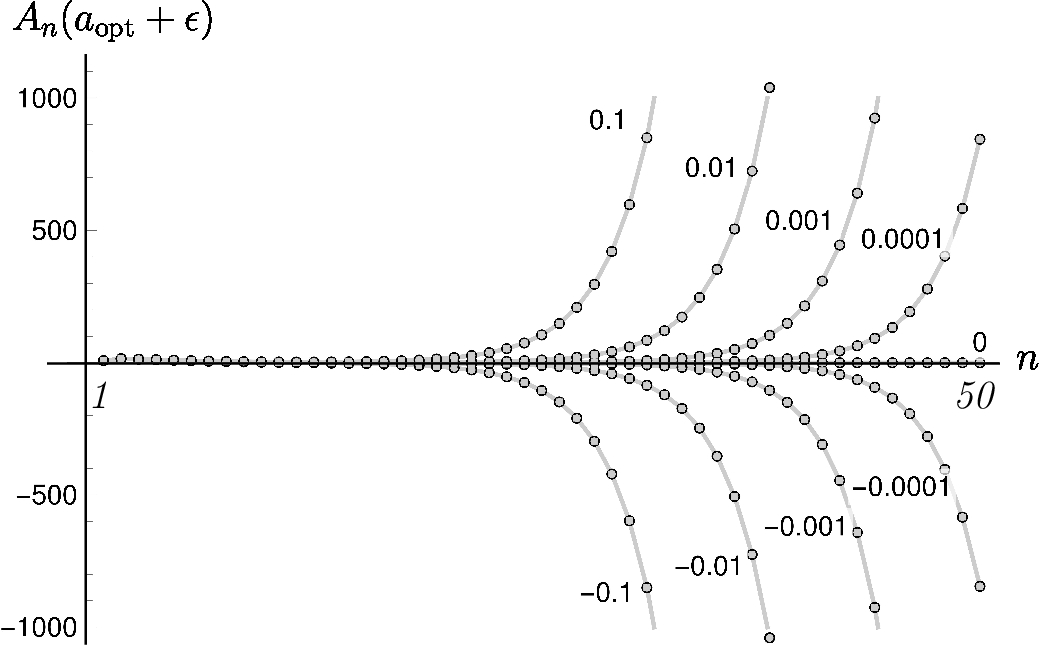}{3.5,-7.1}{5.6}{scale=.27}
  \caption{\small Rapid growth of the first 50 coefficients in nonconvergent
    algebraic solutions generated by to $a_\opt+\epsilon$, illustrated
    for $\eta_0=0.4$ and $m=2$, with $a_\opt$ approximated by
    $-D_{50}/C_{50}$. (The graphic is truncated: for $\epsilon=.1$, the
    coefficients reach approximately $10^7$. Even at this scale, the
    coefficients for $\epsilon=0$ are virtually indistinguishable from
    the horizontal axis.
    \label{fig:aopt}}
\end{figure}
%%%%%%%%%%%%%%%%%%%%%%%%%%%  

%%%%%%%%%%%%%%%%%%%%%%%%%%%%% 
\begin{figure}[b!]
  \pic{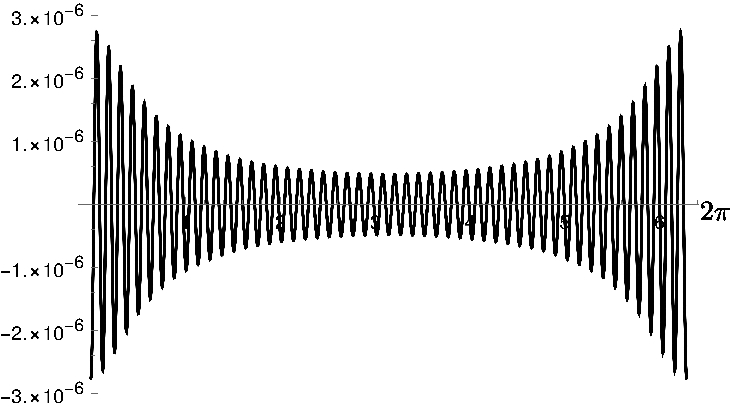}{3.1,-7.}{5.2}{scale=.4}
  \caption{\small Error in solution for Neumann problem for $\eta_0=0.4$, $m=2$ and
    50 terms, distributed over the range $0\le\theta\le2\pi$,
    with $\varphi=0$.)\label{fig:neumannerr}} 
\end{figure}
%%%%%%%%%%%%%%%%%%%%%%%%%%%%%
   
%%%%%%%%%%%%%%%%%%%%%%%%%%%%%
\begin{table}[b!] \centering
$  \begin{array}{l||r@{.}l|r@{.}l|r@{.}l|r@{.}l|}
 N&\multicolumn{2}{c|}{m=1} &\multicolumn{2}{c|}{m=2} \
&\multicolumn{2}{c|}{m=3} &\multicolumn{2}{c|}{m=4}\\\hline
15 &  4&7  &  3&6  &  2&8  &  1&0  \\
20 &  6&7  &  5&5  &  4&5  &  3&7  \\
25 &  8&4  &  7&2  &  6&3  &  5&4  \\
\end{array} 
$
\caption{\small Significant figures in the numerical solution of the
  Neumann problem on the torus showing the increase in accuracy with
  the number of terms. \label{tab:neumannerror}} 
\end{table}
%%%%%%%%%%%%%%%%%%%%%%%%%%%%%

\section{Exterior toroidal domain and toroidal
  shells} \label{Exterior_toroidal_domain_and_toroidal_shells}

\subsection{Exterior domain}

The formula for the normal derivative of an exterior harmonic function
$u\in\har(\dom^*)$ and the solution of the corresponding Neumann
problem are quite analogous to that of the interior domain $\dom$. The
exterior harmonics $\harext{n,m}{\nu,\mu}$ defined by
\eqref{eq:exteriorharmonic} are obtained from the interior harmonics
\eqref{eq:interiorharmonic} by writing $P_{n-1/2}^{m}(\cosh\eta)$ in
place of $Q_{n-1/2}^{m}(\cosh\eta)$ and are orthogonal with the same
weight function \eqref{eq:weight} but applied in $\dom^*$.

These Legendre functions of the first and second kinds satisfy identical
recurrence relationships \cite{Hob1931}. For this reason one finds that
$\nor\harext{n,m}{\nu,\mu}$ is obtained from the formula of Lemma
\ref{lemm:nori} by replacing similarly $Q_{n-1/2}^{m}(\cosh\eta)$ with
$P_{n-1/2}^{m}(\cosh\eta)$. Since the solution to the Dirichlet problem in $\dom^*$
with boundary condition $f$ given by \eqref{eq:fseries} is
\begin{align} 
u=\sum_{n,m,\nu,\mu} a\ind{n,m}{\nu,\mu}\harext{n,m}{\nu,\mu}(\eta,\theta,\varphi),
\end{align}
one finds that the normal derivative of $f$ will be given by equations
\eqref{eq:bfroma} when $q_{n,m}$ is replaced in
\eqref{eq:defrhosigmatau} with
\begin{align}
   p_{n,m} = P_{n-1/2}^m(\cosh\eta_0).
\end{align}
The method we have described is then applicable with no essential
changes for solving the Dirichlet-to-Neumann problem in $\dom^*$. It
is worth noting that parallel to \eqref{eq:qratiolimit} we have
\cite[p.\ 305]{Hob1931} that
\begin{align}
    \lim_{n\to \infty}\frac{p_{n-1,m}}{p_{n,m}}=e^{\eta_{0}}.
\end{align}

\subsection{Toroidal shell}

The results for interior and exterior domains may be combined to solve
the Neumann problem for a toroidal shell.  Let
$\eta_{\rm int}<\eta_{\rm ext}$. Common to an interior and an exterior
domain, one has the toroidal shell
\begin{align*}
  \Omega=  \Omega_{\eta_{\rm int},\eta_{\rm ext}}
  = \Omega_{\eta_{\rm ext}} \cap \Omega_{\eta_{\rm int}}^* .
\end{align*}
A general harmonic function $u$ in
$\Omega_{\eta_{\rm ext},\eta_{\rm int}}$ and continuous in the closure
can be expressed via an integral of its boundary values over
$\partial\Omega_{\eta_{\rm ext},\eta_{\rm int}}$ using the Poisson
kernel for the torus \cite[Ch.\ 1]{HarTri2008}. This integral is the
difference of the integrals over $\partial\Omega_{\eta_{\rm ext}}$ and
$\partial\Omega_{\eta_{\rm int}}$, which give a decomposition
$u=u_0+u_1$ with $u_0\in\har\Omega_{\eta_{\rm int}}$ and
$u_1\in\har\Omega_{\eta_{\rm ext}}^*$. Consequently, we may express $u$
as the sum of two series
\begin{align}
 \label{eq:doubleharmonic}
 u = \sum_{n,m,\nu,\mu}  c\ind{n,m}{\rm{int}\,\nu,\mu}\harint{n,m}{\nu,\mu} +
  \sum_{n,m,\nu,\mu}   c\ind{n,m}{\rm{ext}\,\nu,\mu}\harext{n,m}{\nu,\mu} ,
\end{align}
analogous to the Laurent series for holomorphic functions in an
annular domain in the complex plane, converging uniformly in proper
closed subdomains. (Note, however, that the inner and outer harmonics
together do not form an orthogonal system in
$\Omega_{\eta_{\rm ext},\eta_{\rm int}}$.)

A boundary function $f\colon\partial\Omega\to\R$ is given collectively
by its values for $\eta=\eta_{\rm int}$ and $\eta=\eta_{\rm ext}$
collectively, let us say
\begin{align} \label{eq:doubleboundary}
 f_{\rm int}(\theta,\varphi)= f(\eta_{\rm int},\theta,\varphi) &= 
    \sum a\ind{n,m}{{\rm int}\,\nu,\mu}
    I\ind{n,m}{\nu,\mu}[\eta_{1}] , \nonumber\\
 f_{\rm ext}(\theta,\varphi)=  f(\eta_{\rm ext},\theta,\varphi) &= 
    \sum a\ind{n,m}{{\rm ext}\,\nu,\mu}
     E\ind{n,m}{\nu,\mu}[\eta_{0}]. 
\end{align}
For $u$ to be the solution of the Dirichlet problem for $f$, we
combine \eqref{eq:doubleharmonic} with \eqref{eq:doubleboundary} to find
\begin{align}
q\ind{n,m}{\rm int} c\ind{n,m}{{\rm int}\,\nu,\mu}
  + p\ind{n,m}{\rm int} c\ind{n,m}{{\rm ext}\,\nu,\mu}  &=
   a\ind{n,m}{{\rm int} \,\nu,\mu},  \nonumber\\
    q\ind{n,m}{\rm ext} c\ind{n,m}{{\rm int}\,\nu,\mu}
  + p\ind{n,m}{\rm ext} c\ind{n,m}{{\rm ext}\,\nu,\mu}  &=
   a\ind{n,m}{{\rm ext} \,\nu,\mu},
\end{align}
where
\begin{alignat*}{3}
  q\ind{n,m}{\rm int} &=  Q_{n-1/2}^m(\cosh\eta_{\rm int}) , \quad
  q\ind{n,m}{\rm ext} &\,=\,&  Q_{n-1/2}^m(\cosh\eta_{\rm ext}) , \\
  p\ind{n,m}{\rm int} &=  P_{n-1/2}^m(\cosh\eta_{\rm int}) , \quad
  p\ind{n,m}{\rm ext} &\,=\,&  P_{n-1/2}^m(\cosh\eta_{\rm ext}) .
\end{alignat*}
This might be written symbolically as
\begin{align*}
\begin{pmatrix}
  q^{\rm int} & p^{\rm int} \\
  q^{\rm ext} & p^{\rm ext} 
\end{pmatrix}
\begin{pmatrix}
   c^{\rm int} \\ c^{\rm ext}         
\end{pmatrix}     
=
\begin{pmatrix}
   a^{\rm int} \\ a^{\rm ext}         
\end{pmatrix} .    
\end{align*}
To solve this system, one needs to verify that it is nonsingular.
Instead of a direct verification as in Lemma \ref{lemm:nonzerocoef},
we simply note that if for even one combination of $(n,m,\nu,\mu)$
there were more than one solution, one could easily construct a
Dirichlet problem in the shell $\Omega$ with more than one solution.

We see that $\nor\harint{n,m}{\nu,\mu}\vert_{\partial\Omega_{\rm int}}$ is
obtained from the formula of Lemma \ref{lemm:nori} with $\eta_0$ replaced
with $\eta_{\rm int}$, while
$\nor\harint{n,m}{\nu,\mu}\vert_{\partial\Omega_{\rm ext}}$ is obtained by
using $\eta_{\rm ext}$ instead. The boundary values
$\nor\harext{n,m}{\nu,\mu}\vert_{\partial\Omega_{\rm int}}$ and
$\nor\harext{n,m}{\nu,\mu}\vert_{\partial\Omega_{\rm ext}}$ are then
obtained by replacing $Q_{n-1/2}^m$ with $P_{n-1/2}^m$.
Once we have the harmonic function $u$ as in (\ref{eq:doubleharmonic}),
we have then 
\begin{align*} % \label{eq:nordoubleharmonic}
  \nor u\bigg\vert_{\partial\Omega_{\rm int}} =
  \sum_{n,m,\nu,\mu}  c\ind{n,m}{\rm{int}\,\nu,\mu}
  \nor\harint{n,m}{\nu,\mu} \bigg\vert_{\partial\Omega_{\rm int}}+
  \sum_{n,m,\nu,\mu}   c\ind{n,m}{\rm{ext}\,\nu,\mu}
  \nor\harext{n,m}{\nu,\mu}\bigg\vert_{\partial\Omega_{\rm int}} , \\
  \nor u\bigg\vert_{\partial\Omega_{\rm ext}} =
  \sum_{n,m,\nu,\mu}  c\ind{n,m}{\rm{int}\,\nu,\mu}
  \nor\harint{n,m}{\nu,\mu} \bigg\vert_{\partial\Omega_{\rm ext}}+
  \sum_{n,m,\nu,\mu}   c\ind{n,m}{\rm{ext}\,\nu,\mu}
  \nor\harext{n,m}{\nu,\mu}\bigg\vert_{\partial\Omega_{\rm ext}} .
\end{align*}
When the convergence of the series is absolute, one may apply the same
rearranging and reindexing as described in the proof of Theorem
\ref{theo:bfroma} to obtain the coefficients in the
Dirichlet-to-Neumann mapping $h=\Lambda f$,
\begin{align} \label{eq:defhdouble}
  h(\eta_{\rm int},\theta,\varphi) =  \chcosz \sum  b\ind{n,m}{{\rm int}\,\nu,\mu}  \Phi\ind{n}{\nu}(\theta) \Phi\ind{m}{\mu}(\varphi)  ,\nonumber \\
  h(\eta_{\rm ext},\theta,\varphi) =  \chcosz \sum  b\ind{n,m}{{\rm ext}\,\nu,\mu}
  \Phi\ind{n}{\nu}(\theta) \Phi\ind{m}{\mu}(\varphi).
\end{align}
 
As in the solution of the Neumann problem for the interior domain, the
equations for a fixed value of $(m,\nu,\mu)$ are independent of those
for another value of these indices. They can be solved recursively.
The only difference will be that one must solve a pair of equations at
each step.

\section{Conclusions}

We have presented an approach for studying the Dirichlet-to-Neumann
mapping and solving the Neumann problem for the Laplace operator on a
torus. It is shown how the Dirichlet-to-Neumann mapping may be
expressed by means of certain infinite series based on toroidal
harmonics.  We express the well-known necessary and sufficient
condition for the solvability of the Neumann problem (compatibility
condition), as well as the normalization condition in terms of the
Fourier coefficients. These results show that the Neumann problem
involves an infinite system of linear equations. The solution to the
problem involves a special twist in that the unique value of the free
parameter in this underdetermined linear system which truly gives a
solution cannot be found algebraically. Therefore we express it as a
limit of easily calculated algebraic expressions. Numerical results
are displayed for the accuracy of the algorithm. The paper concludes
showing how the results for interior and exterior domains apply to
solve the Neumann problem for a toroidal shell. The issue of relaxing
the convergence rate requirement on the expansion coefficients is a
thorny problem for the future.

\section{Acknowledgments}

Z.\ Ashtab was supported by CONACyT (Mexico).

%%%%%%%%%%%%%%%%%%%%%%%%%%%%%%%%%%%%%%%%%%%
\newcommand{\authors}[1]{#1}
\newcommand{\booktitle}[1]{#1}
\newcommand{\articletitle}[1]{#1}
\newcommand{\journalname}[1]{#1}
\newcommand{\volnum}[1]{#1}

%%%%%%%%%%%%%%%%%%%%%%%%%%%%%%%%%%%%

\end{document}